\documentclass{amsart}[12pt,article]

\usepackage{amsmath,amssymb,amscd,amsthm,indentfirst}
\usepackage{amsfonts}
\usepackage{cases}
\usepackage[all]{xy}
\usepackage{enumerate}
\usepackage{hyperref}
\usepackage{multicol}
\usepackage{graphicx}
\usepackage{tikz}
\usetikzlibrary{shapes,arrows,calc,positioning}

\newtheorem{thm}{Theorem}[section]
\newtheorem{prop}[thm]{Proposition}
\newtheorem{cor}[thm]{Corollary}
\newtheorem{lem}[thm]{Lemma}

\newtheorem{conjecture}[thm]{Conjecture}

\theoremstyle{definition}
\newtheorem{defn}[thm]{Definition}
\newtheorem{remark}[thm]{Remark}
\newtheorem{finalremark}[thm]{Final Remark}

\newtheorem{exa}[thm]{Example}

\def\Aut{\mathrm{Aut}}

\def\A{\ensuremath{\mathcal{A}}}
\def\B{\ensuremath{\mathcal{B}}}
\def\C{\ensuremath{\mathcal{C}}}
\def\D{\ensuremath{\mathcal{D}}}

\def\I{\ensuremath{\mathcal{I}}}

\def\G{\ensuremath{\mathcal{G}}}

\def\N{\ensuremath{\mathcal{N}}}

\def\P{\ensuremath{\mathcal{P}}}
\def\Q{\ensuremath{\mathcal{Q}}}
\def\SS{\ensuremath{\mathcal{S}}}

\def\FF{\ensuremath{\mathbb{F}}}

\def\ZZ{\ensuremath{\mathbb{Z}}}
\def\QQ{\ensuremath{\mathbb{Q}}}

\def\Out{\operatorname{Out}}

\def\Aut{\mathrm{Aut}}

\newcommand{\Syl}{\operatorname{Syl}\nolimits}

\date{\today}
\title{On Quillen's conjecture for $p$-solvable groups}
\author{Antonio D\'{i}az Ramos}
\address{Departamento de {\'A}lgebra, Geometr{\'\i}a y Topolog{\'\i}a,
Universidad de M{\'a}\-la\-ga, Apdo correos 59, 29080 M{\'a}laga,
Spain.}
\thanks{Supported by MICINN grant RYC-2010-05663. Partially supported by MEC grant MTM2013-41768-P and Junta de Andaluc{\'\i}a grant FQM-213.}
\email{adiazramos@uma.es}

\begin{document}

\begin{abstract}
We give a new proof of Quillen's conjecture for solvable groups via a  geometric and explicit method. For $p$-solvable groups, we provide both a new proof using the Classification of Finite Simple Groups and an asymptotic version without employing it.
\end{abstract}

\maketitle
\section{Introduction}\label{section:Introduction and motivation}
Let $G$ be a finite group, let $p$ be a prime and let $\A_p(G)$ be the poset consisting of the non-trivial elementary abelian $p$-groups of $G$ ordered by inclusion. The homotopy properties of the topological realization $|\A_p(G)|$ were first studied in  \cite{Quillen1978}. There, Quillen introduced the following conjecture, where we denote by $O_p(G)$ the largest normal $p$-subgroup of $G$:
\begin{conjecture}[Quillen's conjecture]\label{conj:Quillen's}
If $|\A_p(G)|$ is contractible then $O_p(G)\neq 1$.
\end{conjecture}
We start discussing some of the known progress on this conjecture as well as we briefly comment on the methods employed so far. If the implication \ref{conj:Quillen's} holds for $G$ we say that \emph{$G$ satisfies $\Q\C$}. Let us introduce the following notion.

\begin{defn} Let $G$ be a finite group of $p$-rank $r$. We say that $G$ has \emph{Quillen dimension at $p$} if $O_p(G)=1\Rightarrow \widetilde H_{r-1}(|\A_p(G)|;\QQ)\neq 0$. 
\end{defn}

This notion was introduced by Ashchbacher and Smith in \cite[p.474]{AS1993} and we denote it by $\Q\D_p$. Note that $r-1$ is the top dimension for which $\widetilde H_*(|\A_p(G)|;\QQ)$ can possibly be non-zero. As contractibility leads to zero homology it is clear that:
\[
\text{$\Q\D_p$ holds for $G$$\Rightarrow$ $\Q\C$ holds for $G$.}
\]
Quillen observed the following:
\begin{equation}\label{equ:Quillenenoughforsolvablecase}
\text{$\Q\D_p$ for all $K\rtimes \FF_p^r$ with $K$ solvable $p'$-group $\Rightarrow$ $\Q\D_p$ for all $G$ solvable,}
\end{equation}
where $\FF_p^r$ is an elementary abelian group of order $p^r$ acting on $K$ and $K\rtimes \FF_p^r$ is their semidirect product. In fact, the left hand side of the implication above was proven by Quillen himself \cite[Corollary 12.2]{Quillen1978}, leading to the solution of the conjecture in the solvable case. This is an inductive proof built on Cohen-Macaulay posets. Alperin supplied  an alternative approach that examines a minimal counterexample via coprime action results \cite[Theorem 8.2.9]{Smith2011}. A further observation of Quillen is:
\begin{equation}\label{equ:Quillenenoughforpsolvablecase}
\text{$\Q\D_p$ for all $K\rtimes \FF_p^r$ with $K$ $p'$-group $\Rightarrow$ $\Q\D_p$ for all $G$ $p$-solvable.}
\end{equation}
Note that the solvability requirement on $K$ has been dropped. In this case, the left hand side, and hence $\Q\C$ for $p$-solvable groups, holds by an extension of the aforementioned argument of Alperin via the Classification of the Finite Simple Groups (CFSG). See \cite[Theorem 1]{Alperin1990}, \cite[Theorem 8.2.12]{Smith2011} and \cite[Theorem 0.5]{AS1993}. %It is also remarkable the combinatorial proof by Hawks and Isaacs of the $p$-solvable case under the additional hypothesis of abelian Sylow $p$-subgroups \cite{HH1988} and Alperin's partial simplification of this argument \cite[Theorem 5.5.4]{Smith2011}.

As contractiblity produces zero homology with any trivial coefficients, it makes sense to include in the above discussion homology with (trivial) coefficients in any abelian group $A$. In this work, we investigate certain  homology classes that we term \emph{constructible classes} and that form a subgroup $\widetilde H^c_{r-1}(|\A_p(K\rtimes \FF_p^r)|;A)$ of $\widetilde H_{r-1}(|\A_p(K\rtimes \FF_p^r)|;A)$, see Definition \ref{defn:constructiblelclass} for details. 
%They are built as follows: For $G$ of $p$-rank $r$, fix $H\leq G$ with $H=\FF_p^r$. Then, for each $a\in A$, we define a chain $Z_{H,a}$ in top-dimension $(r-1)$ for $|\A_p(H)|$ (Section \ref{section:chainelementaryabelianpgroups}). For each $G$-conjugate of $H$ one of these chains is selected, and all of them are added up to give rise to a chain $Z_{G,a_\cdot}$ in top-dimension $(r-1)$ for $|\A_p(G)|$ (Section \ref{section:chainsemidirectproduct}). Here $a_\cdot=(a_S)_{S\in {}^GH}$ and hence $Z_{G,a_\cdot}$ depends on some maximal elements in $\A_p(G)$, i.e., we have removed a subdivision. The chain $Z_{G,a_\cdot}$ is a cycle in $H_{r-1}(|\A_p(G)|;A)$ if the variables $a_\cdot$ satisfy a set of linear equations. We only describe the construction for the case of interest $G=K\rtimes H$ with $H=\FF_p^r$ and $K$ a $p'$-group. In this case, we have $a_\cdot=(a_S)_{S\in \Syl_p(G)}$. 

\begin{defn}\label{def:QDpA} Consider the semidirect product $K\rtimes \FF_p^r$ with  $K$ a $p'$-group and let $A$ be any abelian group. We say that $K\rtimes \FF_p^r$ has \emph{constructible Quillen dimension at $p$ and $A$} if $O_p(K\rtimes \FF_p^r)=1\Rightarrow \widetilde H^c_{r-1}(|\A_p(K\rtimes \FF_p^r)|;A)\neq 0$. 
\end{defn}
We denote this condition by $\Q\D^{c,A}_p$. By the expression $\Q\D^c_p$ we mean that $\Q\D^{c,A}_p$ holds for some abelian group $A$. Note that for the semidirect product $K\rtimes \FF_p^r$:
\begin{equation}\label{equ:O_p=1sameasfaithfulaction}
O_p(K\rtimes \FF_p^r)=1\Leftrightarrow\text{ the action of $\FF_p^r$ on $K$ is faithful.}
\end{equation}
As before, we have that:
\begin{equation}\label{equ:enoughforsolvablecase}
\text{$\Q\D^c_p$ for all $K\rtimes \FF_p^r$ with $K$ solvable $p'$-group $\Rightarrow$ $\Q\C$ for all $G$ solvable}
\end{equation}
and
\begin{equation}\label{equ:enoughforpsolvablecase}
\text{$\Q\D^c_p$ for all $K\rtimes \FF_p^r$ with $K$ $p'$-group $\Rightarrow$ $\Q\C$ for all $G$ $p$-solvable.}
\end{equation}

A constructible class for the group $K\rtimes \FF_p^r$ is defined by a choice of an element of $A$ for each Sylow $p$-subgroup of this group, $a_\cdot=(a_S)_{S\in \Syl_p(K\rtimes \FF_p^r)}$, such that certain homogeneous linear equations are satisfied. Non-trivial solutions of these equations gives rise to non-trivial elements in $\widetilde H^c_{r-1}(|\A_p(K\rtimes \FF_p^r)|;A)$. So, roughly speaking, we have removed a subdivision when computing constructible homology classes. This means that we only need to look at the top two layers of this poset, i.e., Sylow subgroups and their hyperplanes.  We show below that constructible classes suffice to prove Quillen's conjecture in the $p$-solvable case. We start with the following findings. 

\begin{thm}\label{thm:Kisqgroupabelianasymptotic} For the elementary abelian group $\FF_p^r$ or order $p^r$ we have:
\begin{enumerate}[(a)]
\item \label{thm:Kisqgroupabelianasymptotic.qgroup}$\Q\D^{c,\ZZ_q}_p$ holds for $K\rtimes \FF_p^r$ with $K$ a $q$-group and $q\neq p$. 
\item \label{thm:Kisqgroupabelianasymptotic.abelian} $\Q\D^{c,\ZZ_2}_p$ holds for $K\rtimes \FF_p^r$ with $K$ an abelian $p'$-group.
\end{enumerate}
\end{thm}

Here, $\ZZ_D=\ZZ/D\ZZ$.  We succinctly outline the families $a_\cdot$ utilised for this result. For case \eqref{thm:Kisqgroupabelianasymptotic.qgroup}, we choose $a_\cdot$ to take the constant value $1\in \ZZ_q$. Letting $a_\cdot$ to be the characteristic function of an appropriately selected subset of Sylow $p$-subgroups gives \eqref{thm:Kisqgroupabelianasymptotic.abelian}. For the next result, we show that the aforementioned linear equations have at least a non-trivial solution. 

\begin{thm}\label{thm:asymptotic}
$\Q\D^{c,\ZZ}_p$ holds for $K\rtimes \FF_p^r$ with $K$ a $p'$-group if $|K|={q_1}^{e_1}\cdots {q_l}^{e_l}$ satisfies that $r<q_i$ for all $i=1,\ldots,l$.
\end{thm}

In view of the previous discussion, Theorem \ref{thm:asymptotic} is in fact an  asymptotic version of Quillen's conjecture for $p$-solvable groups. If $K$ is solvable, its Fitting subgroup is self-centralizing and has abelian Frattini quotient. These two techniques are of interest here because they preserve faithfulness and because of Equation \eqref{equ:O_p=1sameasfaithfulaction}. Hence, the solvable case can be studied via the abelian case above  together with the following fact that constructible classes behave extremely well with respect to quotients.

\begin{thm}\label{thm:conquotients}
Consider $K\rtimes \FF_p^r$ with $K$ a $p'$-group and let $N\unlhd K$ be an $\FF_p^r$-invariant subgroup of $K$. Let $A$ be any abelian group. If $\widetilde H^c_{r-1}(|\A_p(K/N\rtimes \FF_p^r)|;A)\neq 0$ then $\widetilde H^c_{r-1}(|\A_p(K\rtimes \FF_p^r)|;A)\neq 0$.
\end{thm}

From here and Theorem \ref{thm:Kisqgroupabelianasymptotic}\eqref{thm:Kisqgroupabelianasymptotic.abelian} we obtain the next result.

\begin{thm}\label{thm:QDpcZ2solvable}
$\Q\D^{c,\ZZ_2}_p$ holds for $K\rtimes \FF_p^r$ with $K$ a solvable $p'$-group.
\end{thm}

This result together with Equation \eqref{equ:enoughforsolvablecase} has the following immediate consequence. 
\begin{thm}\label{thm:QCsolvable}
Quillen's conjecture holds for solvable groups.
\end{thm}

For general $p'$-group $K$, the Frattini quotient of its generalized Fitting subgroup is the direct product of an abelian group with simple groups. Theorem \ref{thm:conquotients} again reduces the problem  to building constructible classes on this quotient. As in Theorem \ref{thm:Kisqgroupabelianasymptotic}\eqref{thm:Kisqgroupabelianasymptotic.abelian}, we choose $a_\cdot$ to be the characteristic function of a certain subset of Sylow $p$-subgroups, and we use the CFSG to single out this subset. 

\begin{thm}\label{thm:QDpcZ2p'}
$\Q\D^{c,\ZZ_2}_p$ holds for $K\rtimes \FF_p^r$ with $K$ a $p'$-group.
\end{thm}

Now employing Equation \eqref{equ:enoughforpsolvablecase} we obtain the $p$-solvable case.

\begin{thm}\label{thm:QCpsolvable}
Quillen's conjecture holds for $p$-solvable groups.
\end{thm}

Going back to the initial discussion, we point out that Alperin's arguments for the solvable and $p$-solvable cases are aimed to build a homology sphere in top dimension $r-1$. This sphere is built by appropriately choosing points to form $r$ $0$-spheres, and the considering its join. Here, our constructible homology class in top dimension $r-1$ is made out of conical  pieces of that dimension glued along a graph, see Examples \ref{exa:chainforeagr=3},  \ref{exa:chainforsemidirectr=2} and \ref{exa:c3c3c3c2c2c2}. The weight of our arguments lies in the geometry of the situation, and we do not use neither induction nor minimal  counterexample reasoning. For more details on this comparison, please see Remarks \ref{rmk:liftstoZZjoinofspheres} and \ref{rmk:final}. 

The methods explained here are open in at least a couple of directions: First, one can also define the constructible subgroup $\widetilde H^c_{r-1}(|\A_p(G)|;A)$ for any finite group $G$ of rank $r$ and  any abelian group $A$. Nevertheless, the situation is more complicated as, to start with, the coprime action Propositions \ref{prop:semidirectbasics} and \ref{prop:semidirectbasicsCpir} do not hold in general. Second, it is easy to sharpen Theorem \ref{thm:existenceof2system} in several ways. For instance, the conditions $c_i$ normalizes $C_K(H)$ and $[c_i,c_j]\in C_K(H)$ for all $i$ and $j$, are also enough to have a $2$-system.

\begin{remark}\label{rmk:aposteriorifaithfulaction}
By \cite[Proposition 2.4]{Quillen1978}, which is the reversed implication of Conjecture \ref{conj:Quillen's}, the condition $\widetilde H_*(|\A_p(K\rtimes \FF_p^r)|;A)\neq 0$ forces the action of $\FF_p^r$ on $K$ to be faithful. In particular, if the thesis of Theorem \ref{thm:quotient} holds, we deduce that $\FF_p^r$ acts faithfully on $K$ and on $K/N$.
\end{remark}

\textbf{Acknowledgements:} I am really grateful to Stephen D. Smith for his time and generosity. He read an earlier version of this work and contributed many interesting and clarifying comments, specially those concerning buildings and the relation to Alperin's solution. He also pointed out some flaws in the examples and in the analysis of fixed points in Section \ref{section:Quillen'sconjecture} that are now fixed.

\textbf{Organization of the paper:} In Section \ref{section:Preliminaries} we introduce notation and basic facts about the semidirects products we are interested in. In Section \ref{section:chainelementaryabelianpgroups} we define, for each element $a\in A$, a chain in top dimension $r-1$ for $|\A_p(\FF_p^r)|$. Then, in Section \ref{section:chainsemidirectproduct}, we add up these chains to form a chain in top dimension for $|\A_p(K\rtimes \FF_p^r)|$. We also formulate the linear equations that must be satisfied in order for this chain to become a cycle in $\widetilde H^c_{r-1}(|\A_p(K\rtimes \FF_p^r)|;A)$. Case \eqref{thm:Kisqgroupabelianasymptotic.qgroup} of Theorem \ref{thm:Kisqgroupabelianasymptotic} and Theorem \ref{thm:asymptotic}  are also proven. In Section \ref{section:Dsystemsandquotients}, we introduce some tool needed to prove later in the same section the abelian case Theorem \ref{thm:Kisqgroupabelianasymptotic}\eqref{thm:Kisqgroupabelianasymptotic.abelian}. Afterwards, we prove Theorem \ref{thm:conquotients}. In the final Section \ref{section:Quillen'sconjecture}, we prove Theorems \ref{thm:QDpcZ2solvable}, \ref{thm:QCsolvable}, \ref{thm:QDpcZ2p'} and \ref{thm:QCpsolvable}.

\section{Preliminaries}\label{section:Preliminaries}
Throughout the paper, $p$ denotes a prime and $A$ denotes an abelian group. For a poset $\P$, its realization $|\P|$ is the topological realization of its order complex $\Delta(\P)$, i.e., of the simplicial complex whose $n$-simplices are the compositions $p_0<p_1<\ldots<p_n$ in $\P$. In particular, we compute the homology $\widetilde H_*(|\P|;A)$ via the simplicial complex $C_*(\Delta(\P);A)$. Here, $C_n(\Delta(\P);A)$ has one generator for each $n$-simplex of $\Delta(P)$ and $C_{-1}(\Delta(\P);A)$ has as generator the empty composition $\emptyset$. The differential $d\colon C_n(\Delta(\P);A)\to C_{n-1}(\Delta(\P);A)$ is given as usual by $d=\sum_{i=0}^n (-1)^id_i$, where $d_i$ forgets the $i$-th vertex of a given simplex. If $G$ is a finite group and $p$ is a prime, let $\A_p(G)$  denote the Quillen poset of $G$ at the prime $p$ \cite{Quillen1978}. This poset consists of the non-trivial elementary abelian $p$-subgroups of $G$ ordered by inclusion. We will pay special attention to $\A_p(G)$ for $G$ a semidirect product  $G=K\rtimes H$, where $H$ is an abelian $p$-group acting on the $p'$-group $K$. Note that in this situation $\Syl_p(G)={}^KH=\{kHk^{-1},k\in K\}$ and $|\Syl_p(G)|=|K|/|C_K(H)|$. We introduce the following notation.

\begin{defn}\label{defn:G=KHSKN}
Let $G$ be the semidirect product $G=K\rtimes H$, where $H$ is an abelian $p$-group and $K$ is a $p'$-group. %Let $\N\subseteq \Syl_p(G)$ by any subset of Sylow $p$-subgroups of $G$.
%Then we denote by $\Syl_p^K(G)$ any subset of $K$ such that 
%$\{H^k,k\in \Syl_p^K(G)\}=\Syl_p(G)\text{ and }|\Syl_p^K(G)|=|\Syl_p(G)|$. 
For any subgroup $I\leq G$, we define $\N(I)$ as the subset of Sylow $p$-subgroups of $G$ that contain $I$ and we also set $N(I)=|\N(I)|$.
\end{defn}

The next two results are basic coprime action properties about the semidirect products we are concerned with.

\begin{prop}\label{prop:semidirectbasics}
Let $G$ be a semidirect product $G=K\rtimes H$, with $H$ an abelian $p$-group and $K$ a $p'$-group. Let $I$ be a subgroup of $H$. Then:
\begin{enumerate}
\item \label{prop:semidirectbasicsonlyconjugate}$I$ is the only conjugate of $I$ that lies in $H$.
\item \label{prop:semidirectbasicsnumbersupergroups}$N(I)=|C_K(I)|/|C_K(H)|$.
\item \label{prop:semidirectbasicscentralizerinquotient} If $N\unlhd K$ is $H$-invariant then $C_{K/N}(H)=C_K(H)N/N$.
\end{enumerate}
\end{prop}
\begin{proof}
For the first claim, we reproduce here for convenience the proof in \cite[Lemma 1.2]{HH1988}. So let ${}^gI$ be a conjugate of $I$ that lies in $H$. Then we have $g=kh$ for some $k\in K$ and $h\in H$ and, as $H$ is abelian, we have ${}^gI={}^kI$. Now let $i\in I$. Then $i^{-1}({}^ki)\in H$ and $i^{-1}({}^ki)=k^{i}k^{-1}\in K$. As $K\cap H=\{1\}$ we obtain $i={}^ki$, $k\in C_K(I)$ and ${}^gI={}^kI=I$. For the second claim, consider the action of $C_K(I)$ on $\N(I)$ given by $k\cdot J={}^k J$ for $k\in C_K(I)$ and $J\in \N(I)$. This action is easily seen to be transitive by the previous argument. Moreover, the isotropy group of $J=H$ is exactly $C_K(H)$. The second claim follows. The last item is standard and can be found in \cite[Lemma 2.2(c)]{HH1988}.
\end{proof}

From here on, we will only consider the case with $H$ an elementary abelian $p$-group of rank $r$, $\FF_p^r$. Then we will think of $H$ as an $\FF_p$-vector space whose elements are $r$-tuples of elements from $\FF_p$: $(x_1,\ldots,x_r)\in H$ with $x_i\in \FF_p$. Consequently, we call \emph{hyperplanes} of $H$ those subgroups $I$ of $H$ with $|H:I|=p$. We say that a collection of hyperplanes is \emph{linearly independent} if their corresponding one-dimensional subspaces are independent in the dual vector space $H^*$. 

\begin{prop}\label{prop:semidirectbasicsCpir}
Let $G$ be a semidirect product $G=K\rtimes H$, with $H=\FF_p^r$ and $K$ a $p'$-group. If $H$ acts faithfully on $K$ then there exists $r$ linearly independent hyperplanes of $H$, $H_1,\ldots,H_r$, that satisfy $N(H_i)>1$.
\end{prop}
\begin{proof}
A proof may be found in \cite[Lemma 3.1]{HH1988} but we include here an alternative argument: By \cite[Exercise 8.1]{AS2000}, we have $K=\langle C_K(I)| I\text{ hyperplane of }H\rangle$. Denote by $\I$ the set of hyperplanes $I$ of $H$ that satisfy $C_K(H)\subsetneq C_K(I)$, i.e., that, by Proposition \ref{prop:semidirectbasics}(2), satisfy $N(I)>1$. The intersection $\cap_{I\in \I} I$ acts trivially on $K$ because $K$ is generated by $\{C_K(I)\}_{I\in \I}$. Hence this intersection must be trivial as the action is faithful. Now observe that the dimension of $\cap_{I\in \I} I$ is exactly $r$ minus the maximal number of linearly independent hyperplanes in $\I$.
\end{proof}

For the case $H=\FF_p^r$, we also introduce the notation $[i_1,\ldots,i_l]$ to denote a sequence of elements from $\{1,\ldots,r\}$ with no repetition, and we say that two sequences are equal if they have the same length and the same elements in the same order. Moreover, we denote by $S^r_l$ the set of all such sequences which are of length $l$.  For a sequence $[i_1,\ldots,i_l]\in S^r_l$, we define the following subspace of $H$:
\begin{equation}\label{equ:hyperplanesofH}
H_{[i_1,\ldots,i_l]}=\{(x_1,\dots,x_r)\in H| x_{i_1}=\ldots=x_{i_l}=0\}.
\end{equation}
Thus $H_\emptyset=H$, $H_{[i]}$ is the hyperplane with zero $i$-th coordinate and $H_{[i_1,\ldots,i_l]}=H_{[i_1]}\cap \ldots \cap H_{[i_l]}\cong \FF_p^{r-l}$. In addition, we also have
\begin{equation}\label{equ:Hseq=Hseq'-equivalence}
H_{[i_1,\ldots,i_l]}=H_{[j_1,\ldots,j_l]}\Leftrightarrow \{i_1,\ldots,i_l\}=\{j_1,\ldots,j_l\},
\end{equation}
where on the right-hand side the equality is between (unordered) sets. We define the signature $\epsilon_{[i_1,\ldots,i_l]}$ as  
\begin{equation}\label{equ:defsignature}
\epsilon_{[i_1,\ldots,i_l]}=(-1)^{n+m},
\end{equation}
where $n$ is the number of transpositions we need to apply to the sequence $[i_1,\ldots,i_l]$ to arrange it in increasing order $[j_1,\ldots,j_l]$, and $m$ is the number of entries in which $[j_1,\ldots,j_l]$ differ from $[1,\ldots,l]$. It is immediate that the parity of $n$ does not depend on the particular choice of transpositions and so $\epsilon_{[i_1,\ldots,i_l]}$ is well defined. As an example consider $[1,4,3]\in S^4_3$. Then we have $\epsilon_{[1,4,3]}=(-1)^{1+2}=-1$. We will need the following property of the signature $\epsilon_\cdot$:

\begin{lem}\label{lem:signatureoppositesigns}
Let $[i_1,\ldots,i_{r-1}]$ and $[j_1,\ldots,j_{r-1}]$ be sequences in $S^r_{r-1}$ such that $[i_1,\ldots,i_{r-2}]=[j_1,\ldots,j_{r-2}]$ and $i_{r-1}\neq j_{r-1}$. Then $\epsilon_{[i_1,\ldots,i_{r-1}]}+\epsilon_{[j_1,\ldots,j_{r-1}]}=0$.
\end{lem}
\begin{proof}
Write $\epsilon_{[i_1,\ldots,i_{r-1}]}=(-1)^{n_i+m_i}$ and $\epsilon_{[j_1,\ldots,j_{r-1}]}=(-1)^{n_j+m_j}$ as in \eqref{equ:defsignature}.
Without loss of generality we may assume that $i_{r-1}<j_{r-1}$ and so we can order the set $\{1,\ldots,r\}$ as follows:
\begin{equation}\label{equ:signatureproperty}
1<2<\ldots<i_{r-1}<\ldots<j_{r-1}<\ldots<r.
\end{equation}
Denote by $n$ the number of transpositions needed to arrange $[i_1,\ldots,i_{r-2}]$ in increasing order. Then, from \eqref{equ:signatureproperty}, $n_i=n+r-1-i_{r-1}$ and $n_j=n+r-j_{r-1}$. It is also clear from \eqref{equ:signatureproperty} that $m_i=r-j_{r-1}$ and that $m_j=r-i_{r-1}$. Hence we get $(n_j+m_j)-(n_i+m_i)=1$ and we are done.
\end{proof}

\section{A chain for elementary abelian $p$-groups}
\label{section:chainelementaryabelianpgroups}

Let $H=\FF^r_p$ be an elementary abelian $p$-subgroup of rank $r$ and let $A$ be an abelian group. We consider the poset $\A_p(H)$ of dimension $r-1$ and its order complex $\Delta(\A_p(H))$. We shall explicitly define a chain in the abelian group $C_{r-1}(\Delta(\A_p(H));A)$, i.e., a chain in top dimension $r-1$. We start defining, for $[i_1,\ldots,i_l]\in S^r_l$, the following $l$-simplex in $\Delta(\A_p(H))$:
\begin{equation}\label{equ:sigmasimplex}
\sigma_{[i_1,\ldots,i_l]}=H_{[i_1,\ldots,i_l]}<H_{[i_1,\ldots,i_{l-1}]}<\ldots<H_{[i_1,i_2]}<H_{[i_1]}<H.
\end{equation}
\begin{remark}
Note that, by repeated application of \eqref{equ:Hseq=Hseq'-equivalence}, we have that 
\[
\sigma_{[i_1,\ldots,i_l]}=\sigma_{[j_1,\ldots,i_j]}\Leftrightarrow [i_1,\ldots,i_l]=[j_1,\ldots,j_l].
\]
%and that
%\[
%\tau_{[i_1,\ldots,i_l]}=\tau_{[j_1,\ldots,i_j]}\Leftrightarrow [i_1,\ldots,i_l]=[j_1,\ldots,j_l].
%\]
\end{remark}
Now, for any $a\in A$, we define the following chain in $C_{r-1}(\Delta(\A_p(H));A)$ :
\begin{equation}\label{equ:Zrpi}
Z_{H,a}=\sum_{[i_1,\ldots,i_{r-1}]\in S^r_{r-1}}{\epsilon_{[i_1,\ldots,i_{r-1}]}}\cdot a\cdot \sigma_{[i_1,\ldots,i_{r-1}]}.
\end{equation}
The next proposition exhibits the key feature of the chain $Z_{H,a}$. 
\begin{prop}\label{prop:differentiakzrpi}
The differential $d(Z_{H,a})\in C_{r-2}(\Delta(\A_p(H));A)$ is given by
$$
d(Z_{H,a})=(-1)^{r-1}\sum_{[i_1,\ldots,i_{r-1}]\in S^r_{r-1}} {\epsilon_{[i_1,\ldots,i_{r-1}]}}\cdot a\cdot \tau_{[i_1,\ldots,i_{r-1}]},
$$
where 
$$
\tau_{[i_1,\ldots,i_{r-1}]}=H_{[i_1,\ldots,i_{r-1}]}<H_{[i_1,\ldots,i_{r-2}]}<\ldots<H_{[i_1,i_2]}<H_{[i_1]}.
$$
\end{prop}
\begin{remark}
Again by \eqref{equ:Hseq=Hseq'-equivalence}, we have  
%\[
%\sigma_{[i_1,\ldots,i_l]}=\sigma_{[j_1,\ldots,i_j]}\Leftrightarrow [i_1,\ldots,i_l]=[j_1,\ldots,j_l].
%\]
%and that
$
\tau_{[i_1,\ldots,i_l]}=\tau_{[j_1,\ldots,i_j]}\Leftrightarrow [i_1,\ldots,i_l]=[j_1,\ldots,j_l].
$
\end{remark}
\begin{proof}
The differential $d\colon C_{r-1}(\Delta(\A_p(H));A)\to C_{r-2}(\Delta(\A_p(H));A)$ is given by the alternate sum $\sum_{i=0}^{r-1} (-1)^i d_i$. Note that, by the rank of the subspaces of $H$, $d_i(\sigma_{[i_1,\ldots,i_{r-1}]})=d_j(\sigma_{[j_1,\ldots,j_{r-1}]})\Rightarrow i=j$. We show below that $d_i(Z_{H,a})=0$ for $0\leq i\leq r-2$. Then the theorem follows. We start considering $d_0$. Note that 
$$
d_0(\sigma_{[i_1,\ldots,i_{r-1}]})=H_{[i_1,\ldots,i_{r-2}]}<\ldots<H_{[i_1,i_2]}<H_{[i_1]}<H,
$$
and that, by repeated application of \eqref{equ:Hseq=Hseq'-equivalence}, we have
$$
d_0(\sigma_{[i_1,\ldots,i_{r-1}]})=d_0(\sigma_{[j_1,\ldots,j_{r-1}]})\Leftrightarrow [i_1,\ldots,i_{r-2}]=[j_1,\ldots,j_{r-2}].
$$
Assume that such condition holds for the sequences $[i_1,\ldots,i_{r-1}]$ and $[j_1,\ldots,j_{r-1}]$. Then they must be identical but for the last terms $i_{r-1}$ and $j_{r-1}$. Moreover, if the given sequences are different, we must have $i_{r-1}\neq j_{r-1}$. Then, by Lemma \ref{lem:signatureoppositesigns}, $\epsilon_{[i_1,\ldots,i_{r-1}]}$ and $\epsilon_{[j_1,\ldots,j_{r-1}]}$ have opposite signs and the corresponding summands cancel each other in the expression for $d_0(Z_{H,a})$. Hence $d_0(Z_{H,a})=0$. 

Now we consider $d_i$ for $0<i\leq r-2$. In this case we have: 
$$
d_i(\sigma_{[i_1,\ldots,i_{r-1}]})=H_{[i_1,\ldots,i_{r-1}]}<\ldots< \hat H_{[i_1,\ldots,i_{r-1-i}]}<\ldots<H_{[i_1,i_2]}<H_{[i_1]}<H,
$$
where the term $H_{[i_1,\ldots,i_{r-1-i}]}$ is missing. Again by \eqref{equ:Hseq=Hseq'-equivalence}, if the sequences $[i_1,\ldots,i_{r-1}]$ and $[j_1,\ldots,j_{r-1}]$ have the same differential $d_i$, then they are identical but for the entries $i_{r-i-1},i_{r-i}$ and $j_{r-i-1},j_{r-i}$, and these entries satisfy $\{i_{r-i-1},i_{r-i}\}=\{j_{r-i-1},j_{r-i}\}$. So, if the two given sequences are not the same one, then they differ by a transposition. Hence $\epsilon_{[i_1,\ldots,i_{r-1}]}$ and $\epsilon_{[j_1,\ldots,j_{r-1}]}$ have opposite signs and their corresponding summands cancel each other in the expression for $d_i(Z_{H,a})$. Thus $d_i(Z_{H,a})=0$.  
\end{proof}

%\begin{exa}\label{exa:chainforeagr=2}
%For $r=2$ we have $S^2_1=\{[1],[2]\}$ and the signatures are $+1$ and $-1$ respectively. Inside $\Delta(A_p(C_p^2))$ we have the following cover relations involving the two $2$-simplexes $\sigma_{[1]}$ and $\sigma_{[2]}$:
%{\footnotesize
%\[
%\xymatrix@R=0pt@C=60pt{
%&H\ar@{-}[]+L;[ld]+R\ar@{-}[]+L;[ld]+R\ar@{-}[]+L;[ldd]+R\\
%H_{[1]}<H&H_{[1]}\ar@{-}[]+L;[l]+R\\
%H_{[2]}<H&H_{[2]}\ar@{-}[]+L;[l]+R
%}
%\]
%}
%So $Z_{H,a}$ can be thought of as a segment with endpoints $H_{[1]}$ and $H_{[2]}$.
%\end{exa}

\begin{exa}\label{exa:chainforeagr=3}
For $r=3$ we have $S^3_2=\{[1,2],[2,1],[1,3],[3,1],[2,3],[3,2]\}$ and the signatures are $+1,-1,-1,+1,+1$ and $-1$  respectively. Within $\Delta(A_p(\FF_p^3))$ we have the following cover relations involving the six $2$-simplexes $\sigma_{[i_1,i_2]}$, $[i_1,i_2]\in S^3_2$:
{\footnotesize
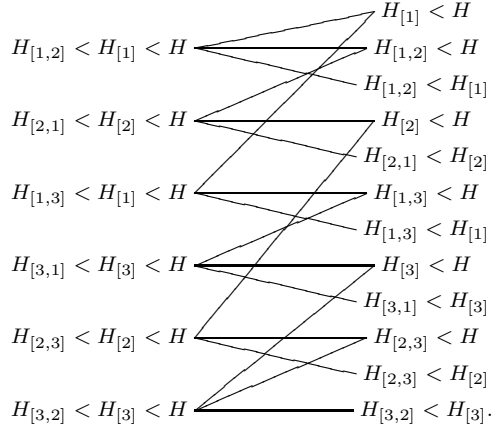
\[
\xymatrix@R=0pt@C=60pt{
&H_{[1]}<H\ar@{-}[]+L;[ld]+R\ar@{-}[]+L;[lddddd]+R\\
H_{[1,2]}<H_{[1]}<H&H_{[1,2]}<H\ar@{-}[]+L;[l]+R\ar@{-}[]+L;[ldd]+R\\
&H_{[1,2]}<H_{[1]}\ar@{-}[]+L;[lu]+R\\
H_{[2,1]}<H_{[2]}<H&H_{[2]}<H\ar@{-}[]+L;[l]+R\ar@{-}[]+L;[ldddddd]+R\\
&H_{[2,1]}<H_{[2]}\ar@{-}[]+L;[lu]+R\\
H_{[1,3]}<H_{[1]}<H&H_{[1,3]}<H\ar@{-}[]+L;[l]+R\ar@{-}[]+L;[ldd]+R\\
&H_{[1,3]}<H_{[1]}\ar@{-}[]+L;[lu]+R\\
H_{[3,1]}<H_{[3]}<H&H_{[3]}<H\ar@{-}[]+L;[l]+R\ar@{-}[]+L;[ldddd]+R\\
&H_{[3,1]}<H_{[3]}\ar@{-}[]+L;[lu]+R\\
H_{[2,3]}<H_{[2]}<H&H_{[2,3]}<H\ar@{-}[]+L;[l]+R\ar@{-}[]+L;[ldd]+R\\
&H_{[2,3]}<H_{[2]}\ar@{-}[]+L;[lu]+R\\
H_{[3,2]}<H_{[3]}<H&H_{[3,2]}<H_{[3]}.\ar@{-}[]+L;[l]+R
}
\]
}
{
\begin{minipage}{0.6\textwidth}
The different signs for the signatures cause $Z_{H,a}$ to have boundary $d(Z_{H,a})$ supported on the six $1$-simplexes $\tau_{[i_1,i_2]}$ with $[i_1,i_2]\in S^3_2$. So the chain $Z_{H,a}$ has support in the right-hand side cone-shaped $2$-dimensional simplicial complex.
\end{minipage}%
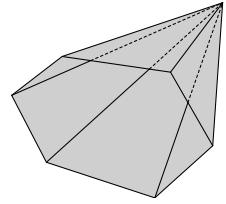
\begin{minipage}{0.4\textwidth}
\centering
\definecolor{aqaqaq}{rgb}{0.6274509803921569,0.6274509803921569,0.6274509803921569}
\begin{tikzpicture}[scale=0.55,line cap=round,line join=round,>=triangle 45,x=1.0cm,y=1.0cm]
\clip(0,-4.3) rectangle (6,1.5);
\fill[line width=0.pt,color=aqaqaq,fill=aqaqaq,fill opacity=0.5] (5.86,0.42) -- (1.98,-0.88) -- (0.76,-1.82) -- (1.6,-3.44) -- (4.22,-4.3) -- (5.62,-3.04) -- cycle;
\draw (1.6,-3.44)-- (4.22,-4.3);
\draw (4.22,-4.3)-- (5.62,-3.04);
\draw (5.62,-3.04)-- (4.6,-1.26);
\draw (4.6,-1.26)-- (1.98,-0.88);
\draw (1.98,-0.88)-- (0.76,-1.82);
\draw (0.76,-1.82)-- (1.6,-3.44);
\draw (5.86,0.42)-- (1.98,-0.88);
\draw (4.6,-1.26)-- (5.86,0.42);
\draw (5.62,-3.04)-- (5.86,0.42);
\draw (5.020999275887038,-1.9946850108616958)-- (4.22,-4.3);
\draw [dash pattern=on 1pt off 1pt] (5.020999275887038,-1.9946850108616958)-- (5.86,0.42);
\draw (4.087880668257756,-1.1857231503579952)-- (1.6,-3.44);
\draw [dash pattern=on 1pt off 1pt] (4.087880668257756,-1.1857231503579952)-- (5.86,0.42);
\draw (2.6717487318747755,-0.9803299687451964)-- (0.76,-1.82);
\draw [dash pattern=on 1pt off 1pt] (2.6717487318747755,-0.9803299687451964)-- (5.86,0.42);
\end{tikzpicture}
\end{minipage}
}
\end{exa}

\begin{remark}
The figure above suggest that $Z_{H,a}$ is a sort of cone construction. In fact, it is straightforward that $Z_{H,a}$ is the join of the vertex $H$ with the standard apartment arising from the basis of the one-dimensional subspaces $\FF_pv_i$, where the vector $v_i=(0,\ldots,0,1,0,\ldots,0)$ has the $1$ in the $i$-th position for $i=1,\ldots,r$.
\end{remark}

\section{A chain for semidirect products}
\label{section:chainsemidirectproduct}

Let $H=\FF^r_p$ be an elementary abelian $p$-subgroup of rank $r$ that acts on the $p'$-group $K$ and consider the semidirect product $G=K\rtimes H$. We consider the poset $\A_p(G)$ of dimension $r-1$ and its order complex $\Delta(\A_p(G))$.  We shall define a chain $C_{r-1}(\Delta(\A_p(G));A)$ in top dimension $r-1$ using the chain defined in Section \ref{section:chainelementaryabelianpgroups}. For each Sylow $p$-subgroup $S\in \Syl_p(G)$ choose $k_S\in K$ with $S={}^{k_S} H$ and choose an element $a_S$ of the abelian group $A$. Then consider the conjugated chain
\begin{equation}
Z_{S,a_S}:={}^{k_S}(Z_{H,a_S}),
\end{equation}
where $Z_{H,a_S}\in C_{r-1}(\Delta(\A_p(H));A)\subseteq C_{r-1}(\Delta(\A_p(G));A)$ was defined in Equation \eqref{equ:Zrpi} and $k_S$ acts by conjugation on $C_{r-1}(\Delta(\A_p(G));A)$ and hence transforms $Z_{H,a_S}$ somewhere inside this chain space. The next lemma shows that $Z_{S,a_S}$ does not depend on the conjugating element $k_S$:

\begin{lem}
Let $k_1,k_2\in K$ with ${}^{k_1}H={}^{k_2}H$ and let $a\in A$. Then ${}^{k_1}(Z_{H,a})={}^{k_2}(Z_{H,a})$.
\end{lem}
\begin{proof}
By Equations \eqref{equ:Zrpi} and \eqref{equ:sigmasimplex}, it is enough to see that 
$$
{}^{k_1}H_{[i_1,\ldots,i_{r-1}]}<{}^{k_1}H_{[i_1,\ldots,i_{r-2}]}<\ldots<{}^{k_1}H_{[i_1,i_2]}<{}^{k_1}H_{[i_1]}<{}^{k_1}H
$$
and
$$
{}^{k_2}H_{[i_1,\ldots,i_{r-1}]}<{}^{k_2}H_{[i_1,\ldots,i_{r-2}]}<\ldots<{}^{k_2}H_{[i_1,i_2]}<{}^{k_2}H_{[i_1]}<{}^{k_2}H
$$
are equal for each sequence $[i_1,\ldots,i_{r-1}]$ of length $r-1$. But note that for any $1\leq l\leq r-1$, the subgroups ${}^{k_1}H_{[i_1,\ldots,i_l]}$ and ${}^{k_2}H_{[i_1,\ldots,i_l]}$ are $K$-conjugates of $H_{[i_1,\ldots,i_l]}$ lying in ${}^{k_1}H={}^{k_2}H$. Hence, by Proposition \ref{prop:semidirectbasics}\eqref{prop:semidirectbasicsonlyconjugate}, these two subgroups must be equal. The result follows. 
\end{proof}

Now, for $a_\cdot=(a_S)_{S\in \Syl_p(G)}$ elements of $A$, and $(k_S)_{S\in \Syl_p(G)}$ elements from $K$ satisfying ${}^{k_S}H=S$, we define the following element of $C_{r-1}(\Delta(\A_p(G));A)$:
\begin{equation}\label{equ:Z}
Z_{G,a_\cdot}=\sum_{S\in \Syl_p(G)} Z_{S,a_S}=\sum_{S\in \Syl_p(G)} {}^{k_S}(Z_{H,a_S}).
\end{equation}
The lemma above shows that $Z_{G,a_\cdot}$ does not depend on the particular choice of the elements $k_S$'s.
\begin{prop}\label{prop:dZG}
The differential $d(Z_{G,a_\cdot})\in C_{r-2}(\Delta(\A_p(G));A)$ is given by
%$$
%d(Z_G)=(-1)^{r-1}\bigoplus_{((i_1,\ldots,i_{r-1}),k)\in T} (-1)^{\epsilon_{(i_1,\ldots,i_{r-1})}}\cdot z\cdot N(H^k_{(i_1)})\cdot \tau^k_{(i_1,\ldots,i_{r-1})},
%$$
$$
d(Z_{G,a_\cdot})=(-1)^{r-1}\sum_{([i_1,\ldots,i_{r-1}],S)\in T} {\epsilon_{[i_1,\ldots,i_{r-1}]}}\cdot \Big(\sum_{S'\in  \N({}^{k_S}H_{[i_1]})} a_{S'}\Big)\cdot {}^{k_S}\tau_{[i_1,\ldots,i_{r-1}]},
$$
where $T$ is certain subset of $S^r_{r-1}\times \Syl_p(G)$ and 
$$
{}^{k_S}\tau_{[i_1,\ldots,i_{r-1}]}={}^{k_S}H_{[i_1,\ldots,i_{r-1}]}<{}^{k_S}H_{[i_1,\ldots,i_{r-2}[}<\ldots<{}^{k_S}H_{[i_1,i_2]}<{}^{k_S}H_{[i_1]}.
$$
\end{prop}
\begin{proof}
We have $d(Z_{G,a_\cdot})=\sum_{S\in \Syl_p(G)} d({}^{k_S}Z_{H,a_s})$ and, by Proposition \ref{prop:differentiakzrpi}, the differential of ${}^{k_S}Z_{H,a_s}$ is given by
$$
d({}^{k_S}Z_{H,a_s})=(-1)^{r-1}\sum_{[i_1,\ldots,i_{r-1}]\in S^r_{r-1}} {\epsilon_{[i_1,\ldots,i_{r-1}]}}\cdot a_S\cdot {}^{k_S}\tau_{[i_1,\ldots,i_{r-1}]},
$$
where $\tau_{[i_1,\ldots,i_{r-1}]}=H_{[i_1,\ldots,i_{r-1}]}<H_{[i_1,\ldots,i_{r-2}]}<\ldots<H_{[i_1,i_2]}<H_{[i_1]}$.

To study $d(Z_{G,a_\cdot})$, we define on the set $S^r_{r-1}\times \Syl_p(G)$ the following equivalence relation:
$$
([i_1,\ldots,i_{r-1}],S)\sim ([j_1,\ldots,j_{r-1}],S')\Leftrightarrow {}^{k_S}\tau_{[i_1,\ldots,i_{r-1}]}={}^{k_{S'}}\tau_{[j_1,\ldots,j_{r-1}]}.
$$
We also set $T\subseteq S^r_{r-1}\times \Syl_p(G)$ to be any set of representatives for the equivalence classes of $\sim$. Now we fix $([i_1,\ldots,i_{r-1}),S)\in T$ and study its $\sim$-equivalence class. So assume that $([i_1,\ldots,i_{r-1}],S)\sim ([j_1,\ldots,j_{r-1}],S')$. Then, in the first place, we have ${}^{k_S}H_{[i_1]}={}^{k_{S'}}H_{[j_1]}< {}^{k_{S'}}H$. In addition, for any $1\leq l\leq r-1$, the subgroup ${}^{{k_{S'}}^{-1}k_S}H_{[i_1,\ldots,i_l]}=H_{[j_1,\ldots,j_l]}$ is a subgroup of $H$ that is $K$-conjugated to $H_{[i_1,\ldots,i_l]}\leq H$. Hence, by Proposition \ref{prop:semidirectbasics}\eqref{prop:semidirectbasicsonlyconjugate}, we must have $H_{[i_1,\ldots,i_l]}=H_{[j_1,\ldots,j_l]}$. Using \eqref{equ:Hseq=Hseq'-equivalence} we obtain then that $[i_1,\ldots,i_{r-1}]=[j_1,\ldots,j_{r-1}]$. So we have proven the left to right implication of the next claim:
\begin{align*}
([i_1,\ldots,i_{r-1}],S)\sim ([j_1,\ldots,j_{r-1}],S')\Leftrightarrow [i_1,\ldots,i_{r-1}]&=[j_1,\ldots,j_{r-1}]\\
\text{ and }{}^{k_S}H_{[i_1]}&<{}^{k_{S'}}H,
\end{align*}
and the right to left implication is a direct consequence of Proposition \ref{prop:semidirectbasics}\eqref{prop:semidirectbasicsonlyconjugate} again. It turns out then that the signature is constant along the $\sim$-equivalence class of $([i_1,\ldots,i_{r-1}],S)$ and the formula in the statement follows.
% and that the size of this class is exactly the number $N(H_{(i_1)}^{k_1})$ of Sylow $p$-subgroups of $G$ that contain $H_{(i_1)}^{k_1}$. The formula in the statement follows.
\end{proof}

Note that the chain $Z_{G,a_\cdot}$ is a cycle, i.e., $d(Z_{G,a_\cdot})=0$, if and only if
\begin{equation}\label{equ:chainZiscycle}
\sum_{S\in  \N({}^{k}H_{[i]})} a_{S}=0 
\end{equation}
for all $i\in \{1,\ldots,r\}$ and all $K$-conjugates ${}^kH_{[i]}$ of $H_{[i]}$, and it is non-trivial if and only if
\begin{equation}\label{equ:chainZisnontrivial}
a_S\neq 0\text{ for some $S\in \Syl_p(G)$.}
\end{equation}

\begin{defn}\label{defn:constructiblelclass}
If equations \eqref{equ:chainZiscycle} are satisfied we say that $[Z_{G,a_\cdot}]$ is a \emph{constructible class} in $\widetilde H_{r-1}(|\A_p(K\rtimes H)|;A)$.
\end{defn}

It is immediate that constructible classes form a subgroup $\widetilde H^c_{r-1}(|\A_p(K\rtimes H)|;A)$ of $\widetilde H_{r-1}(|\A_p(K\rtimes H)|;A)$ as claimed in the introduction. For $r=1$, i.e., for $H=\FF_p$, and considering the ``hyperplane'' $\emptyset \in \Delta(\A_p(K\rtimes H))$, Equations \eqref{equ:chainZiscycle} and \eqref{equ:chainZisnontrivial} just say that $\widetilde H^c_0(|\A_p(K\rtimes H)|;A)=\widetilde H_0(|\A_p(K\rtimes H)|;A)$. 

\begin{remark}\label{remark:z_kconstant}
If for some $a\in A$ we choose $a_S=a$ for all $S\in \Syl_p(G)$ then the Equation for acyclicity of $Z_{G,a_\cdot}$, (\ref{equ:chainZiscycle}), simplifies to:
\begin{equation}\label{equ:chainZiscycleconstantvalue}
N(H_{[i]})\cdot a=0 
\end{equation}
for all $i\in\{1,\ldots,r\}$.
\end{remark}
\begin{cor}\label{cor:gcdgreaterthan1}
Let $G$ be a semidirect product $G=K\rtimes H$ with $H=\FF_p^r$ and $K$ a $p'$-group. Assume   there are $r$ linearly independent hyperplanes $H_1,\ldots,H_r$ such that $D>1$ divides the numbers $N(H_1),\ldots,N(H_r)$. Then  $\widetilde H^c_{r-1}(|\A_p(G)|;\ZZ_D)\neq 0$.
\end{cor}
\begin{proof}
By a suitable change of basis in the $\FF_p$-vector space $H$, we may assume that $H_i=H_{[i]}$ for $i=1,\ldots,r$. Set $A=\ZZ/D\ZZ=\ZZ_D$. Then, if we choose $a=1\in A$ as in Remark \ref{remark:z_kconstant}, we have $N(H_i)\cdot 1=0 \mod D$ for all $i$  and  it follows that $Z_{G,a_\cdot}$ is a non-trivial homology class in $\widetilde H^c_{r-1}(|\A_p(G)|;A)$.
\end{proof}
\begin{cor}\label{cor:Kisrhogroup}
Let $G$ be a semidirect product $G=K\rtimes H$ with $H=\FF_p^r$ acting faithfully on the $q$-group $K$ for a prime $q\neq p$. Then $\widetilde H^c_{r-1}(|\A_p(G)|;\ZZ_q)\neq 0$.
\end{cor}
\begin{proof}
By Proposition \ref{prop:semidirectbasicsCpir}, there exist $r$ linearly independent hyperplanes $H_1,\ldots,H_r$ that satisfy $N(H_i)>1$. Being $K$ a $q$-group, this implies that $1<q$ divides $N(H_i)$ for all $i$. Then, by Corollary \ref{cor:gcdgreaterthan1}, $\widetilde H^c_{r-1}(|\A_p(G)|;\ZZ_q)\neq 0$.
\end{proof}

This proves Theorem \ref{thm:Kisqgroupabelianasymptotic}\eqref{thm:Kisqgroupabelianasymptotic.qgroup} of the introduction. Equations \eqref{equ:chainZiscycle} form a system of homogeneous linear equations over the abelian group $A$. It consists of $\sum_{i=1}^r |K|/|C_K(H_{[i]})|$ equations over $|K|/|C_K(H)|$ variables and the entries of the matrix associated to the system are either $0$ or $1$. Assume that $A$ is a ring which is a Principal Ideal Domain (P.I.D.). Then, if there are fewer equations than variables:
\begin{equation}\label{equation:linearsystem_undetermined}
\sum_{i=1}^r \frac{|K|}{|C_K(H_{[i]})|}<\frac{|K|}{|C_K(H)|}\Leftrightarrow\sum_{i=1}^r \frac{1}{|N(H_{[i]})|}<1,
\end{equation}
the Smith-Normal form of the matrix associated to the system shows that there exists at least a non-trivial solution.

\begin{cor}\label{cor:asymptotic}
Let $G$ be a semidirect product $G=K\rtimes H$ with $H=\FF_p^r$ acting faithfully on the $p'$-group $K$ with $|K|=q_1^{e_1}\cdots q_l^{e_l}$. If $r<q_i$ for $i=1,\ldots,l$, then $\widetilde H^c_{r-1}(|\A_p(G)|;\ZZ)\neq 0$.
\end{cor}
\begin{proof}
After a suitable reordering, we may assume that $q_1<\ldots<q_r$ and hence the hypothesis on $r$ reads as $r<q_1$. As the integers $\ZZ$ form a P.I.D., we must check that Equation \eqref{equation:linearsystem_undetermined} holds. By Proposition \ref{prop:semidirectbasicsCpir} and a suitable change of basis in the $\FF_p$-vector space $H$ we may assume that $H_i=H_{[i]}$ for $i=1,\ldots,r$ with $|N(H_{[i]})|>1$.  Because these numbers divide $|K|$ we have $|N(H_{[i]})|\geq q_1$ and $
\sum_{i=1}^r \frac{1}{|N(H_{[i]})|}\leq \frac{r}{q_1}<1$.
\end{proof}

This result proves Theorem \ref{thm:asymptotic} of the introduction.  In view of Equations \eqref{equ:chainZiscycle}, it is worth considering the following graph.

\begin{defn}\label{def:graphKH}
For the semidirect product $G=K\rtimes H$ with $K$ a $p'$-group and $H\cong \FF_p^r$, let $\G$ be the graph with vertices all $K$-conjugates of $H$ together with all $K$-conjugates of the hyperplanes $H_{[i]}$, and with edges all the inclusions among them.
\end{defn}

By Proposition \ref{prop:semidirectbasics}\eqref{prop:semidirectbasicsonlyconjugate}, the vertices that are $K$-conjugates of $H$ have degree equal to $r$. A vertex which is $K$-conjugate to $H_{[i]}$ has degree $N(H_{[i]})$. Moreover, Equations \eqref{equ:chainZiscycle} can be interpreted as assigning the value $a_S$ with $S\in \Syl_p(G)$ to the vertex $S\in \Syl_p(G)$ and, for each vertex ${}^k H_{[i]}$ with $k\in K$, equalizing to zero the sum of the values associated to its neighbours.

\begin{exa}\label{exa:chainforsemidirectr=2}
Consider $G=C_3\times C_3\rtimes \FF_2\times \FF_2$ with the generators acting by $(x_1,x_2)\mapsto (-x_1,x_2)$ and $(x_1,x_2)\mapsto (x_1,-x_2)$ respectively. Set $H_1=\FF_2\times 0$ and $H_2=0\times \FF_2$. Then we have $C_K(H_1)=0\times C_3$, $C_K(H_2)=C_3\times 0$, $C_K(H)=1$ and $N(H_1)=N(H_2)=3$. The graph $\G$ is shown below, where upper-case letters correspond to conjugates of $H$ and lower-case letters to conjugates of hyperplanes. By Corollary \ref{cor:Kisrhogroup}, there exists a non-trivial homology class $Z_G$ in the top dimension homology group $H^c_1(\Delta(\A_2(G));\ZZ_3)$. Moreover, by Corollary \ref{cor:asymptotic}, as there are $6$ equations and $9$ variables, we also have $\widetilde H^c_1(\Delta(\A_2(G));\ZZ)\neq 0$.

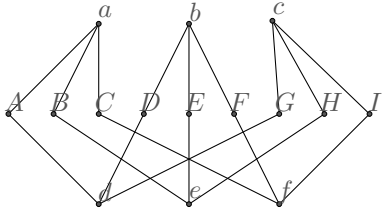
\centering
\definecolor{qqqqff}{rgb}{0.3333333333333333,0.3333333333333333,0.3333333333333333}
\begin{tikzpicture}[scale=0.6,line cap=round,line join=round,>=triangle 45,x=1.0cm,y=1.0cm]
\clip(-1,2) rectangle (9,8);
\draw (2.,7.)-- (0.,5.);
\draw (0.,5.)-- (2.,3.);
\draw (2.,7.)-- (1.,5.);
\draw (1.,5.)-- (4.,3.);
\draw (2.,7.)-- (2.,5.);
\draw (2.,5.)-- (6.,3.);
\draw (4.,7.)-- (3.,5.);
\draw (3.,5.)-- (2.,3.);
\draw (4.,7.)-- (4.,5.);
\draw (4.,5.)-- (4.,3.);
\draw (4.,7.)-- (5.,5.);
\draw (5.,5.)-- (6.,3.);
\draw (5.848559003862463,7.066314453177751)-- (6.,5.);
\draw (6.,5.)-- (2.,3.);
\draw (5.848559003862463,7.066314453177751)-- (7.,5.);
\draw (7.,5.)-- (4.,3.);
\draw (5.848559003862463,7.066314453177751)-- (8.,5.);
\draw (8.,5.)-- (6.,3.);
%\begin{scriptsize}
\draw [fill=qqqqff] (0.,5.) circle (1.5pt);
\draw[color=qqqqff] (0.12498176809010648,5.2598698106582376) node {$A$};
\draw [fill=qqqqff] (1.,5.) circle (1.5pt);
\draw[color=qqqqff] (1.132787726548362,5.2598698106582376) node {$B$};
\draw [fill=qqqqff] (2.,5.) circle (1.5pt);
\draw[color=qqqqff] (2.140593685006617,5.2598698106582376) node {$C$};
\draw [fill=qqqqff] (3.,5.) circle (1.5pt);
\draw[color=qqqqff] (3.1293844367015096,5.2598698106582376) node {$D$};
\draw [fill=qqqqff] (4.,5.) circle (1.5pt);
\draw[color=qqqqff] (4.137190395159765,5.2598698106582376) node {$E$};
\draw [fill=qqqqff] (5.,5.) circle (1.5pt);
\draw[color=qqqqff] (5.125981146854658,5.2598698106582376) node {$F$};
\draw [fill=qqqqff] (6.,5.) circle (1.5pt);
\draw[color=qqqqff] (6.133787105312913,5.2598698106582376) node {$G$};
\draw [fill=qqqqff] (7.,5.) circle (1.5pt);
\draw[color=qqqqff] (7.141593063771168,5.2598698106582376) node {$H$};
\draw [fill=qqqqff] (8.,5.) circle (1.5pt);
\draw[color=qqqqff] (8.130383815466061,5.2598698106582376) node {$I$};
\draw [fill=qqqqff] (2.,7.) circle (1.5pt);
\draw[color=qqqqff] (2.140593685006617,7.275481727574747) node {$a$};
\draw [fill=qqqqff] (4.,7.) circle (1.5pt);
\draw[color=qqqqff] (4.137190395159765,7.275481727574747) node {$b$};
\draw [fill=qqqqff] (5.848559003862463,7.066314453177751) circle (1.5pt);
\draw[color=qqqqff] (5.981665451206006,7.332527347864837) node {$c$};
\draw [fill=qqqqff] (2.,3.) circle (1.5pt);
\draw[color=qqqqff] (2.140593685006617,3.2632731005050912) node {$d$};
\draw [fill=qqqqff] (4.,3.) circle (1.5pt);
\draw[color=qqqqff] (4.137190395159765,3.2632731005050912) node {$e$};
\draw [fill=qqqqff] (6.,3.) circle (1.5pt);
\draw[color=qqqqff] (6.133787105312913,3.2632731005050912) node {$f$};
%\end{scriptsize}
\end{tikzpicture}
\end{exa}
\section{$D$-systems and quotients}
\label{section:Dsystemsandquotients}

Under some hypothesis, we have constructed above non-trivial classes in top homology that are constant solutions (Remark \ref{remark:z_kconstant}) to the Equations \eqref{equ:chainZiscycle}. In this section, we construct solutions which are characteristic functions. We first specialize Definition \ref{defn:G=KHSKN}.

\begin{defn}
Let $G$ be a semidirect product $G=K\rtimes H$ with $H=\FF_p^r$ and $K$ a $p'$-group. 
Let $\SS\subseteq \Syl_p(G)$ be a non-empty subset of Sylow $p$-subgroups. For any subgroup $I\leq G$, we define $\N_\SS(I)$ as the subset of Sylow $p$-subgroups of $\SS$ that contain $I$ and we also set $N_\SS(I)=|\N_\SS(I)|$. If $D>1$ is an integer, we say that $\SS$ is a \emph{$D$-system} if $D$ divides $N_\SS({}^kH_{[i]})$ for all $i$ and all $k\in K$.
\end{defn} 

\begin{thm}\label{thm:Dsystem}
Let $G$ be a semidirect product $G=K\rtimes H$ with $H=\FF_p^r$ and $K$ a $p'$-group. If $\SS$ is a $D$-system then $\widetilde H^c_{r-1}(|\A_p(G)|;\ZZ_D)\neq 0$. 
\end{thm}
\begin{proof}
Form the chain $Z_{G,a_\cdot}$ as in Equation \eqref{equ:Z}, where we choose $a_S=1$ for $S\in \SS$ and $0$ otherwise. This is a not trivial chain and the Equations \eqref{equ:chainZiscycle} for $Z_{G,a_\cdot}$ being a cycle become
\[
N_\SS({}^kH_{[i]})=0 \mod D,
\]
for all $i\in \{1,\ldots,r\}$ and all $K$-conjugates ${}^kH_{[i]}$ of $H_{[i]}$. This holds by hypothesis.
\end{proof}

Next we formulate a group theoretical condition that implies the existence of a $2$-system.

\begin{thm}\label{thm:existenceof2system}
Let $G$ be a semidirect product $G=K\rtimes H$ with $H=\FF_p^r$ and $K$ a $p'$-group. Assume there are elements $c_i\in C_K(H_{[i]})\setminus C_K(H)$ for $i=1,\ldots,r$ such $[c_i,c_j]=1$ for all $i$ and $j$. Then there exists a $2$-system for $G$ and hence $\widetilde H^c_{r-1}(|\A_p(G)|;\ZZ_2)\neq 0$.
\end{thm}
\begin{proof}
Let $\SS=\{{}^{c_1^{\delta_1}c_2^{\delta_2}\ldots c_r^{\delta_r}}H\text{, with $\delta_i\in \{0,1\}$ }\} \subseteq \Syl_p(G)$. We claim that $\SS$ is a $2$-system: consider ${}^kH_{[i]}$ for some $i\in \{1,\ldots,r\}$ and some $k\in K$. If $N_S({}^kH_{[i]})=0$ then there is nothing to prove. Assume then that the set $\N_\SS({}^kH_{[i]})$ is not empty and choose one of its elements: ${}^{c_1^{\delta_1}c_2^{\delta_2}\ldots c_r^{\delta_r}}H\in \N_\SS({}^kH_{[i]})$ for some values $\delta_i$'s in $\{0,1\}$. Then by Proposition \ref{prop:semidirectbasics}\eqref{prop:semidirectbasicsonlyconjugate} we have 
\[
{}^kH_{[i]}={}^{c_1^{\delta_1}c_2^{\delta_2}\ldots
 c_r^{\delta_r}}H_{[i]}<{}^{c_1^{\delta_1}c_2^{\delta_2}\ldots c_r^{\delta_r}}H.
\]
Assume first that $\delta_i=0$. Because $[c_i,c_j]=1$ for all $j$, we have:
\[
{}^{c_i}({}^{c_1^{\delta_1}\ldots c_{i-1}^{\delta_{i-1}}c_{i+1}^{\delta_{i+1}}\ldots c_r^{\delta_r}}H)={}^{c_1^{\delta_1}\ldots c_{i-1}^{\delta_{i-1}}c_ic_{i+1}^{\delta_{i+1}}\ldots c_r^{\delta_r}}H={}^{c_1^{\delta_1}\ldots c_{i-1}^{\delta_{i-1}}c_{i+1}^{\delta_{i+1}}\ldots c_r^{\delta_r}c_i}H.
\]
From here we deduce that ${}^{c_i}({}^{c_1^{\delta_1}\ldots c_{i-1}^{\delta_{i-1}}c_{i+1}^{\delta_{i+1}}\ldots c_r^{\delta_r}}H_{[i]})$ is a subgroup of the last two subgroups in the display above. So, by Proposition \ref{prop:semidirectbasics}\eqref{prop:semidirectbasicsonlyconjugate} again, we must have:
\[
{}^{c_i}({}^{c_1^{\delta_1}\ldots c_{i-1}^{\delta_{i-1}}c_{i+1}^{\delta_{i+1}}\ldots c_r^{\delta_r}}H_{[i]})={}^{c_1^{\delta_1}\ldots c_{i-1}^{\delta_{i-1}}c_ic_{i+1}^{\delta_{i+1}}\ldots c_r^{\delta_r}}H_{[i]}={}^{c_1^{\delta_1}\ldots c_{i-1}^{\delta_{i-1}}c_{i+1}^{\delta_{i+1}}\ldots c_r^{\delta_r}c_i}H_{[i]}.
\]
As $c_i$ centralizes $H_{[i]}$, the right-hand side group in the equation above is equal to ${}^{c_1^{\delta_1}\ldots c_{i-1}^{\delta_{i-1}}c_{i+1}^{\delta_{i+1}}\ldots c_r^{\delta_r}}H_{[i]}$. Hence, ${}^{c_i}({}^{c_1^{\delta_1}\ldots c_{i-1}^{\delta_{i-1}}c_{i+1}^{\delta_{i+1}}\ldots c_r^{\delta_r}}H)\in \N_\SS({}^kH_{[i]})$. Moreover, this $K$-conjugate of $H$ cannot be equal to ${}^{c_1^{\delta_1}\ldots c_{i-1}^{\delta_{i-1}}c_{i+1}^{\delta_{i+1}}\ldots c_r^{\delta_r}}H$ because then we would obtain $c_i\in C_K(H)$, a contradiction. So, we have shown, in the case $\delta_i=0$, that 
\[
{}^{c_1^{\delta_1}\ldots c_i^{\delta_i}\ldots c_r^{\delta_r}}H\in \N_\SS({}^kH_{[i]})\Rightarrow {}^{c_1^{\delta_1}\ldots c_i^{1-\delta_i}\ldots c_r^{\delta_r}}H\in \N_\SS({}^kH_{[i]}),
\]
and both $K$-conjugates of $H$ are different. The case $\delta_i=1$ is proven analogously using conjugation by $c_i^{-1}$ instead of by $c_i$. It follows that $N_\SS({}^kH_{[i]})$ is an even number.
\end{proof}

\begin{remark}\label{rmk:liftstoZZjoinofspheres}
The set $\SS$ defined in the proof of Theorem \ref{thm:existenceof2system} has size $2^r$: Consider non-trivial elements $h_i\in \cap_{j\neq i} H_{[j]}$. Then $H=\langle h_1,\ldots,h_r\rangle$. Moreover, if we have ${}^{c_1^{\epsilon_1}\ldots c_r^{\epsilon_r}}H=H$ with $\epsilon_i\in \{-1,0,1\}$, we obtain, by Proposition  \ref{prop:semidirectbasics}\eqref{prop:semidirectbasicsonlyconjugate}, that ${}^{c_i^{\epsilon_i}}h_i=h_i$ for all $i$. If there exists $i_0$ with $\epsilon_{i_0}\neq 0$, we deduce that  $[c_{i_0},h_{i_0}]=1$ and then, as $H=\langle H_{[i_0]},h_{i_0}\rangle$, that $c_{i_0}\in C_K(H)$, a contradiction.

Consider then the  chain  $Z_{G,b_\cdot}$, where $b_S=(-1)^{\sum_i \delta_i}$ if $S={}^{c_1^{\delta_1}c_2^{\delta_2}\ldots c_r^{\delta_r}}H\in \SS$ and $b_S=0$ otherwise. It represents a homology class  $[Z_{G,b_\cdot}]\in \widetilde H^c_{r-1}(|\A_p(G)|;\ZZ)$ which is a lift of  $[Z_{G,a_\cdot}]\in \widetilde H^c_{r-1}(|\A_p(G)|;\ZZ_2)$. In fact, $Z_{G,b_\cdot}$ is exactly the cycle corresponding to the join of the $r$ $0$-spheres $\{h_1,{}^{c_1}h_1\}, \ldots, \{h_r,{}^{c_r}h_r\}$.
\end{remark}

\begin{cor}\label{cor:abelian}
Let $G$ be a semidirect product $G=K\rtimes H$ with $H=\FF_p^r$ acting faithfully on the abelian $p'$-group $K$. Then there exists a $2$-system for $G$ and hence $H^c_{r-1}(|\A_p(G)|;\ZZ_2)\neq 0$.
\end{cor}
\begin{proof}
By Proposition \ref{prop:semidirectbasicsCpir}, there exist $r$ linearly independent hyperplanes $H_1,\ldots,H_r$ that satisfy $N(H_i)=|C_K(H_i)|/|C_K(H)|>1$. By a suitable change of basis in the $\FF_p$-vector space $H$, we may assume that $H_i=H_{[i]}$ for $i=1,\ldots,r$. As $K$ is abelian, any choice of elements $c_i\in C_K(H_{[i]})\setminus C_K(H)\neq \emptyset$ for $i=1,\ldots,r$ satisfies the hypothesis of Theorem \ref{thm:existenceof2system}.
\end{proof}

This prove Theorem \ref{thm:Kisqgroupabelianasymptotic}\eqref{thm:Kisqgroupabelianasymptotic.abelian}. We finish this section by proving Theorem \ref{thm:conquotients} of the introduction.

\begin{thm}\label{thm:quotient}
Let $G$ be a semidirect product $G=K\rtimes H$ with $H=\FF_p^r$ and $K$ a $p'$-group and let $N\unlhd K$ be a normal $H$-invariant subgroup of $K$. Consider the quotient $\bar G=G/N$ and let $A$ be an abelian group. Then every class $[Z_{\bar G,\bar a_\cdot}]\in H^c_{r-1}(|\A_p(\bar G)|;A)$ can be lifted to a class  $[Z_{G,a}]\in \widetilde H^c_{r-1}(|\A_p(G)|;A)$. Moreover, a non-trivial class lifts to a non-trivial class.
\end{thm}
\begin{proof}
Consider the projection homomorphism $\pi\colon G\to \bar G$. Because its kernel is the $p'$-order subgroup $N$, there is an induced surjective application:
\[
\Syl_p(\pi)\colon \Syl_p(G)\to \Syl_p(\bar G).
\]
Moreover, for $\pi(S)\in \Syl_p(\bar G)$, we have $\Syl_p(\pi)^{-1}( \pi(S))=\{{}^nS,n\in N\}$ by Proposition \ref{prop:semidirectbasics}\eqref{prop:semidirectbasicscentralizerinquotient}. We have the collection of  elements of $A$, $\bar a_\cdot=({\bar a}_{\bar S})_{\bar S\in \Syl_p(\bar G)}$, and we define  $a_\cdot=(a_S)_{S\in \Syl_p(G)}$ by $a_S={\bar a}_{\pi(S)}$. Now choose a $K$-conjugate ${}^kH_{[i]}$ of $H_{[i]}$. Then the equality
\[
\Syl_p(\pi)(\N({}^kH_{[i]}))=\N({}^{\pi(k)}\pi(H_{[i]}))
\]
follows from Proposition \ref{prop:semidirectbasics}\eqref{prop:semidirectbasicscentralizerinquotient}. 
Equation \eqref{equ:chainZiscycle} for $Z_{G,a}$ at ${}^kH_{[i]}$ become:
\[
\sum_{S\in  \N({}^{k}H_{[i]})} a_{S}=\sum_{\pi(S)\in  \N({}^{\pi(k)}\pi(H_{[i]}))}\hspace{10pt}\sum_{S'\in  \N({}^{k}H_{[i]})\cap \Syl_p^{-1}(\pi)(\pi(S))} {\bar a}_{\pi(S)}=0.
\]
\end{proof}
Note  that $\N({}^{k}H_{[i]})\cap \Syl_p^{-1}(\pi)(\pi(S))=\{{}^nS,n\in N\cap C_K({}^{k}H_{[i]})\}$ and that this set has order $m=|C_N({}^{k}H_{[i]})|/|C_N(S)|$. As $N$ is normal in $K$, the denominator $|C_N(S)|$ does not depend on the Sylow $p$-subgroup $S\in \N({}^{k}H_{[i]})$ and it is clear that $m$ does not depend neither. Hence, we may rewrite the equation above as 
\[
m\cdot \sum_{\pi(S)\in  \N({}^{p(k)}\pi(H_{[i]}))}{\bar a}_{\pi(S)}=0,
\]
and this equation holds because $Z_{\bar G,\bar a}$ is a cycle. It is immediate from the construction that non-trivial classes lifts to non-trivial classes.

\begin{exa}\label{exa:c3c3c3c2c2c2}
Consider $G=C_5\times C_5\times C_5 \rtimes \FF_2\times \FF_2\times \FF_2$ with the generators acting by sending $(x_1,x_2,x_3)$ to $(-x_1,x_2,x_3)$, $(x_1,-x_2,x_3)$ and $(x_1,x_2,-x_3)$ respectively. Set $H_1=0\times \FF_2\times \FF_2$, $H_2=\FF_2\times 0\times \FF_2$ and $H_3=\FF_2\times \FF_2\times 0$. Then we have $C_K(H_1)=C_5\times 0\times 0$, $C_K(H_2)=0\times C_5\times 0$, $C_K(H_3)=0\times 0\times C_5$, $N(H_i)=25$ and $C_K(H)=1$. Although Corollaries \ref{cor:Kisrhogroup} and \ref{cor:asymptotic} apply in this case, we focus on the construction in Corollary \ref{cor:abelian}. Set $c_1=(1,0,0)$, $c_2=(0,1,0)$ and $c_3=(0,0,1)$. Then the elements $c_i\in C_K(H_i)\setminus C_K(H)$ satisfy the hypothesis of Theorem \ref{thm:existenceof2system}. The $2$-system constructed in the proof of this result is exactly $\SS=\{H,{}^{c_1}H,{}^{c_2}H,{}^{c_3}H,{}^{c_1c_2}H,{}^{c_1c_3}H,{}^{c_2c_3}H,{}^{c_1c_2c_3}H\}$ and has size $8$. The subgraph of the graph $\G$ (Definition \ref{def:graphKH}) containing $\SS$ and its neighbours is the following.
{\small
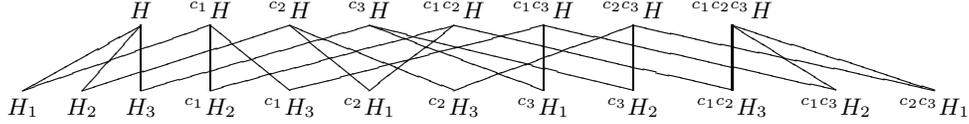
\[
\xymatrix@C=5pt{
&&H\ar@{-}[]+D;[dll]+U\ar@{-}[]+D;[dl]+U\ar@{-}[]+D;[d]+U
&{}^{c_1}H\ar@{-}[]+D;[dlll]+U\ar@{-}[]+D;[d]+U\ar@{-}[]+D;[dr]+U
&{}^{c_2}H\ar@{-}[]+D;[dlll]+U\ar@{-}[]+D;[dr]+U\ar@{-}[]+D;[drr]+U
&{}^{c_3}H\ar@{-}[]+D;[dlll]+U\ar@{-}[]+D;[drr]+U\ar@{-}[]+D;[drrr]+U
&{}^{c_1c_2}H\ar@{-}[]+D;[dlll]+U\ar@{-}[]+D;[dl]+U\ar@{-}[]+D;[drrr]+U
&{}^{c_1c_3}H\ar@{-}[]+D;[dlll]+U\ar@{-}[]+D;[d]+U\ar@{-}[]+D;[drrr]+U
&{}^{c_2c_3}H\ar@{-}[]+D;[dll]+U\ar@{-}[]+D;[d]+U\ar@{-}[]+D;[drrr]+U
&{}^{c_1c_2c_3}H\ar@{-}[]+D;[d]+U\ar@{-}[]+D;[dr]+U\ar@{-}[]+D;[drr]+U\\
H_1&H_2&H_3&{}^{c_1}H_2&{}^{c_1}H_3&{}^{c_2}H_1&{}^{c_2}H_3&{}^{c_3}H_1&{}^{c_3}H_2&{}^{c_1c_2}H_3&{}^{c_1c_3}H_2&{}^{c_2c_3}H_1
}
\]
}
\begin{minipage}{0.6\textwidth}
The homology class $[Z_{G,a_\cdot}]\in \widetilde H^c_2(|\A_p(G)|;\ZZ_2)$ has support in the $2$-dimensional simplicial complex on the right-hand side. The eight faces of the hollow octahedron are in bijection with the set $\SS$. For each element of $\SS$, its corresponding face is decomposed as in Example \ref{exa:chainforeagr=3}, that is, via its barycentric subdivision (shown in one of the faces of the figure).
\end{minipage}
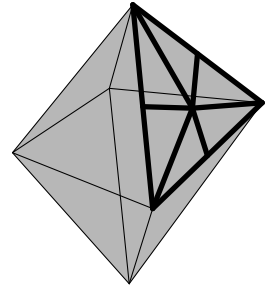
\begin{minipage}{0.4\textwidth}
\centering
\definecolor{aqaqaq}{rgb}{0.6274509803921569,0.6274509803921569,0.6274509803921569}
\begin{tikzpicture}[scale=0.4,line cap=round,line join=round,>=triangle 45,x=1.0cm,y=1.0cm]
\clip(4,2) rectangle (14, 13);
\fill[line width=0.pt,color=aqaqaq,fill=aqaqaq,fill opacity=0.5] (9.12,12.28) -- (5.12,7.36) -- (9.78,5.5) -- cycle;
\fill[line width=0.pt,color=aqaqaq,fill=aqaqaq,fill opacity=0.5] (9.12,12.28) -- (5.12,7.36) -- (8.34,9.5) -- cycle;
\fill[line width=0.pt,color=aqaqaq,fill=aqaqaq,fill opacity=0.5] (9.12,12.28) -- (9.78,5.5) -- (13.42,9.02) -- cycle;
\fill[line width=0.pt,color=aqaqaq,fill=aqaqaq,fill opacity=0.5] (9.12,12.28) -- (8.34,9.5) -- (13.42,9.02) -- cycle;
\fill[line width=0.pt,color=aqaqaq,fill=aqaqaq,fill opacity=0.5] (8.34,9.5) -- (5.12,7.36) -- (9.,3.) -- cycle;
\fill[line width=0.pt,color=aqaqaq,fill=aqaqaq,fill opacity=0.5] (5.12,7.36) -- (9.78,5.5) -- (9.,3.) -- cycle;
\fill[line width=0.pt,color=aqaqaq,fill=aqaqaq,fill opacity=0.5] (9.,3.) -- (9.78,5.5) -- (13.42,9.02) -- cycle;
\fill[line width=0.pt,color=aqaqaq,fill=aqaqaq,fill opacity=0.5] (9.,3.) -- (8.34,9.5) -- (13.42,9.02) -- cycle;
\draw (9.12,12.28)-- (5.12,7.36);
\draw (5.12,7.36)-- (9.78,5.5);
\draw [line width=2.pt] (9.78,5.5)-- (13.42,9.02);
\draw [line width=2.pt] (9.12,12.28)-- (9.78,5.5);
\draw [line width=2.pt] (9.12,12.28)-- (13.42,9.02);
\draw (9.12,12.28)-- (8.34,9.5);
\draw (8.34,9.5)-- (5.12,7.36);
\draw (8.34,9.5)-- (13.42,9.02);
\draw (5.12,7.36)-- (9.,3.);
\draw (9.,3.)-- (9.78,5.5);
\draw (9.,3.)-- (13.42,9.02);
\draw (9.,3.)-- (8.34,9.5);
\draw [line width=2.pt] (9.45,8.89)-- (11.1,8.84);
\draw [line width=2.pt] (11.1,8.84)-- (11.27,10.65);
\draw [line width=2.pt] (11.1,8.84)-- (11.6,7.26);
\draw [line width=2.pt] (11.1,8.84)-- (9.12,12.28);
\draw [line width=2.pt] (11.1,8.84)-- (9.78,5.5);
\draw [line width=2.pt] (11.1,8.84)-- (13.42,9.02);
\end{tikzpicture}
\end{minipage}

\end{exa}

\section{Quillen's conjecture}
\label{section:Quillen'sconjecture}

In this section we prove the results of the introduction directly related to Quillen's conjecture. We start with Theorem \ref{thm:QDpcZ2solvable}.

\begin{thm}\label{thm:QDpcZ2solvableproven}
$\Q\D^{c,\ZZ_2}_p$ holds for $K\rtimes H$ with $H=\FF_p^r$ and $K$ a solvable $p'$-group.
\end{thm}
\begin{proof}
As $O_p(K\rtimes H)=1$ we have that the action of $H$ on $K$ is faithful. Because $K\rtimes H$ is solvable, its Fitting group $F(K\rtimes H)=F(K)$ is self-centralizing \cite[31.10]{AS2000} and hence $H$ also acts faithfully on the nilpotent group $F(K)$. Now, the Frattini quotient $F(K)/\Phi(F(K))$ is abelian and, by a result of Burnside \cite[5.1.4]{Gor1980}, $H$ acts faithfully on it. Next, by Corollary \ref{cor:abelian}, $\widetilde H^c_{r-1}(|\A_p(F(K)/\Phi(F(K))\rtimes H)|;\ZZ_2)\neq 0$. Then by Theorem \ref{thm:quotient} we also have $\widetilde H^c_{r-1}(|\A_p(F(K)\rtimes H)|;\ZZ_2)\neq 0$. Finally, the inclusion $\widetilde  H^c_{r-1}(|\A_p(F(K)\rtimes H)|;\ZZ_2)\leq \widetilde H^c_{r-1}(|\A_p(K\rtimes H)|;\ZZ_2)$ finishes the proof.
\end{proof}

Next we prove Theorem  \ref{thm:QCsolvable} of the introduction.

\begin{thm}\label{thm:QCsolvableproven}
Quillen's conjecture holds for solvable groups.
\end{thm}
\begin{proof}
We reproduce Quillen's argument that reduces this result to the configuration of Theorem \ref{thm:QDpcZ2solvableproven}: Let $G$ be a solvable group with $O_p(G)=1$. Let $r$ be the $p$-rank of $G$ and choose $H\in A_p(G)$ of rank $r$. By Hall-Highman \cite[Lemma 1.2.3]{HH1956} we have that $K=O_{p'}(G)$ is self centralizing and hence $H$ acts faithfully on this solvable group $K$. By Theorem \ref{thm:QDpcZ2solvableproven} we have that $\widetilde H^c_{r-1}(|\A_p(K\rtimes H)|;\ZZ_2)\neq 0$. The inclusion $\widetilde H_{r-1}(|\A_p(K\rtimes H)|;\ZZ_2)\leq \widetilde H_{r-1}(|\A_p(G)|;\ZZ_2)$ finishes the proof.
\end{proof}

Next, we focus on the $p$-solvable case. We start describing fixed points for actions on direct products. We say that the group $H$ acts by permutations on the direct product $Y=X_1\times\ldots\times X_m$ if $H$ acts on $Y$ and each $h\in H$ permutes the groups $\{X_1,\ldots,X_m\}$ among themselves. This is the case, for instance, if each $X_i$ is quasisimple, by the Krull-Schmidt Theorem \cite[3.22]{GLSII}.

\begin{lem}\label{prop:nleqm}
Let $H$ be an elementary abelian $p$-group acting by permutations on the $m$-fold direct product $Y=X\times \ldots \times X$  and transitively permuting these components. Then the following hold:
\begin{enumerate}
\item \label{prop:nleqmcentralizerdescription} If $H'=N_H(X)$ and $H''$ satisfies  $H=H'\times H''$ then 
\[
C_Y(H)=\{\prod_{h\in H''} {}^hx\text{ with $x\in C_X(H')$}\}.
\]
\item \label{prop:nleqmcentralizerquotientequivalence}For $I\leq H$, $H'=N_H(X)$,  $I'=N_I(X)$, and $k$ equal to the number of orbits for the permutation action of $I$ on the components of $Y$, we have:
\[
C_Y(I)>C_Y(H)\Leftrightarrow \bigg[ \Big[ k=1 \text{ and }C_X(I')>C_X(H')\Big] \text{ or }\Big[ k>1 \text{ and } C_X(I')>1 \Big]\bigg].
\]
\end{enumerate}
\end{lem}
\begin{proof}
Set $H'=N_H(X)$ and choose $H''\leq H$ such that $H=H'\times H''$. Then we have $Y={}^{h_1}X\times {}^{h_2}X\times\ldots\times {}^{h_m}X$, where $H''=\{h_1=1,h_2,\ldots,h_m\}$. Moreover, $C_Y(H)=C_Y(H')\cap C_Y(H'')$ and we compute these two centralizers separately. For the former, we have $C_Y(H')=\prod_{h\in H''} C_{{}^h X}(H')=\prod_{h\in H''} {}^h C_X(H')$. For the latter, notice that $H''$ regularly permutes the components of $Y$. The map $X\to Y$ given by $x\mapsto \prod_{h\in H''} {}^h x$ is a 
homomorphism which image is exactly $C_Y(H'')$ (cf. \cite[Lemma 3.27]{GLSII}). The expression in \eqref{prop:nleqmcentralizerdescription} follows.

Regarding \eqref{prop:nleqmcentralizerquotientequivalence}, write $Y=Y_1\cdots Y_k$, where each $Y_l$ is an orbit for the permutation action of $I$ on the components of $Y$. Then $C_Y(I)=C_{Y_1}(I)\cdots C_{Y_k}(I)$. For the action of $I$ on each orbit $Y_l$ and with the notation of \eqref{prop:nleqmcentralizerdescription}, we have $I'=I\cap H'$ and we may choose $I''$ and $H''$ such that $I''=I\cap H''$. Moreover, if $J\leq H''$ satisfies $H''=I''\times J$, then we may write $J=\{j_1,\ldots,j_k\}$ with $Y_l$ the $I$-orbit of ${}^{j_l}X$:
\[
	Y_l=\prod_{i\in I''} {}^{i}({}^{j_l}X).
\]
Note that $k=|J|=|H'':I''|=\frac{|H:H'|}{|I:I'|}$. The result follows from inspection of the expressions in \eqref{prop:nleqmcentralizerdescription} for the centralizers $C_Y(H)$ and $C_{Y_l}(I)$ for $l=1,\ldots k$.
\end{proof}

The next result is of independent interest and it is the only place where we use the Classification of the Finite Simple Groups (CFSG).

\begin{prop}\label{prop:nleqmCFSG}
Let $H$ be an elementary abelian $p$-group acting on a direct product $Y$ of $m$ copies of a nonabelian simple $p'$-group $X$ and transitively permuting these components. Let $H_1,\ldots,H_n$ be linearly independent hyperplanes of $H$ such that $C_Y(H_i)>C_Y(H)$ for all $i$. Then there exist elements $c_i\in C_Y(H_i)\setminus C_Y(H)$ such that $[c_i,c_j]=1$ for all $i,j$.
\end{prop}
\begin{proof}
Write $H'=N_H(X)$, $H'_i=N_{H_i}(X)$ and choose $H''$ and $H''_i$ for each $i$ such that $H=H'\times H''$ and $H_i=H_i'\times H_i''$.  As $H'_i\leq H'$ and $|H:H_i|=p$, we have:
\begin{enumerate}
\item either $H'=H'_i$ and $|H''|/|H''_i|=p$, 
\item or $|H':H_i'|=p$ and $|H''|=|H''_i|$.
\end{enumerate}
In case $(1)$, we have $C_X(H')=C_X(H'_i)$ and, by Lemma \ref{prop:nleqm}\eqref{prop:nleqmcentralizerquotientequivalence}, the permutation action of $H_i$ on $Y$ has exactly $p$ orbits and $|C_X(H'_i)|>1$. In case $(2)$, Lemma \ref{prop:nleqm}\eqref{prop:nleqmcentralizerquotientequivalence} implies that $C_X(H'_i)>C_X(H')$. If $C_H(X)=N_H(X)$, then also $C_{H_i}(X)=N_{H_i}(X)$ and $C_X(H')=C_X(H'_i)=X$, a contradiction. Otherwise, $\Aut_H(X)=N_H(X)/C_H(X)$ is a non-trivial group of outer automorphisms of $X$. But:
\begin{align}
&\textit{If $P$ is a $p$-group and $1\neq P\leq \Out(X)$ for a nonabelian simple $p'$-group $X$,}\label{equ:pouterautomorphismsp'nonabeliansimplegroup}\\
&\textit{then $X$ is of Lie type and $P$ is cyclic and consists of field automorphisms.}\nonumber
\end{align}
The proof of the claim \eqref{equ:pouterautomorphismsp'nonabeliansimplegroup} is an  exhaustive check via the CFSG. Details are provided in \cite[Theorem 8.2.12(2)]{Smith2011} and the same argument is used in \cite{AS1993}. Hence, $\Aut_H(X)$ is cyclic of order $p$. Then, either $|C_H(X):C_{H_i}(X)|=|N_{H_i}(X):C_{H_i}(X)|=p$ or these two indexes are equal to $1$. In the former case, $\Aut_{H_i}(X)$ is a non-trivial subgroup of $\Aut_H(X)$ and hence $\Aut_{H_i}(X)=\Aut_H(X)$. Then $H'=H'_i+C_H(X)$ and $C_X(H')=C_X(H'_i)$, contradiction again. So the latter case must hold, i.e., $C_H(X)=C_{H_i}(X)=N_{H_i}(X)$. We deduce that $C_X(H'_i)=X$ and, from \eqref{equ:pouterautomorphismsp'nonabeliansimplegroup}, that $C_X(H')=X^{\sigma}$ for a field automorphism $\sigma$ of $X$ of  order $p$. To sum up, we conclude that for each hyperplane $H_i$:
\begin{enumerate}
\item either $H_i$ has $p$ permutation orbits on $Y$ and $C_X(H'_i)=C_X(H')>1$,
\item or $H_i$ transitively permutes $Y$ and $C_X(H'_i)=X>C_X(H')=X^\sigma$.
\end{enumerate}
Now we construct the elements $c_i$ in the statement. If all hyperplanes are of type $(1)$ then pick $c\in C_X(H')$ with $c\neq 1$ and set $c_i=\prod_{h\in H''_i} {}^hc$ for all $i$. Otherwise, some hyperplane is of type $(2)$, $X$ is of Lie type and we notice that:
\begin{equation}
\textit{There exists $c\in X^\sigma$ and $d\in X\setminus X^\sigma$  such that $[c,d]=1$.} \label{equ:elementsinXandXsigma}\nonumber
\end{equation}
The elements $c$ and $d$ may be chosen as elements in the maximal torus of $X$ satisfying $\sigma(c)=c$ and $\sigma(d)\neq d$. For $H_i$ of type $(1)$, set $c_i=\prod_{h\in H''_i} {}^hc$. For $H_i$ of type $(2)$, set  $c_i=\prod_{h\in H''} {}^h d$. These elements satisfy $[c_i,c_j]=1$ for all $i,j$.
\end{proof}

\begin{remark}
Under the hypotheses of Proposition \ref{prop:nleqmCFSG}, it also holds that $n\leq m$.
%If $H=\FF_p^r$ then $m=p^{r-rank(N_H(X))}$. For each $i=1,\ldots,n$, let $h_i$ be an element of $H$ such that $H=\langle H_i,h_i\rangle$ and $h_i\in \cap_{j\neq i} H_j$. Then the subspace  $I=\langle h_i,i=1,\ldots,n\rangle$ has rank $n$. If $\sum_i \alpha_ih_i=h\in C_H(X)$ is non zero, then we can write $h_{i_0}=h-\sum_{i\neq i_0} \frac{\alpha_i}{\alpha_{i_0}}h_i$ for some index $i_0$. Then we have $[C_Y(H_{i_0}),h_{i_0}]=[C_Y(H_{i_0}),h-\sum_{i\neq i_0} \frac{\alpha_i}{\alpha_{i_0}}h_i]=1$ because $[X,h]=1$ and because $[C_Y(H_{i_0}),h_i]=1$ as $h_i\in H_{i_0}$ for $i\neq i_0$. So we obtain that $C_Y(H)=C_Y(H_{i_0})$, and this contradicts the hypothesis.
%
%Thus, $I\cap C_H(X)=0$ and hence $n+rank(C_H(X))\leq r$. If $C_H(X)=N_H(X)$ we have $n\leq r-rank(N_H(X))\leq p^{r-rank(N_H(X))}$ and we are done. Otherwise, $N_H(X)/C_H(X)$ is a non-trivial group of outer automorphisms of $X$. Going through the CFSG and using that we are in the coprime situation $(|H|,|X|)=1$, it turns out that $N_H(X)/C_H(X)$ is cyclic. This argument is explained in \cite[Theorem 8.2.12(2)]{Smith2011} and it is also used in \cite{AS1993}. Thus, $rank(N_H(X))=rank(C_H(X))+1$, we have:
%\[
%n\leq r-rank(C_H(X))=r-rank(N_H(X))+1\leq p^{r-rank(N_H(X))},
%\]
%and we are done too. Note that for $m=1$, i.e., for $H=N_H(X)$, the rightmost inequality is an equality and we get $n\leq 1$.
\end{remark}

\begin{thm}\label{thm:QDpcZ2p'proven}
$\Q\D^{c,\ZZ_2}_p$ holds for $K\rtimes \FF_p^r$ with $K$ a $p'$-group.
\end{thm}
\begin{proof}
As $O_p(K\rtimes H)=1$ we have that the action of $H$ on $K$ is faithful. Now, the generalized Fitting subgroup $F^*(K\rtimes H)=F^*(K)$ is self-centralizing \cite[31.13]{AS2000} and hence $H$ acts faithfully on it. Here, $F^*(K)=F(K)E(K)$ with $[E(K),F(K)]=1$ and where $E(K)$ is the \emph{layer}, which is a central product of quasisimple groups. Moreover,  By \cite[3.15(v)]{GLSII}, $\Phi(F^*(K))=\Phi(F(K))Z(E(K))$ and by \cite[3.5(v)]{GLSII}, $F(K)\cap E(K)=Z(E(K))$. Then by \cite[3.9(ii)]{GLSII}, it follows that $F^*(K)/\Phi(F^*(K))=A\times B$, where $A$ is elementary abelian and $B$ is a direct product of nonabelian simple $p'$-groups. By Burnside \cite[5.1.4]{Gor1980}, $H$ acts faithfully on $D=A\times B$, and we are going to show that the hypothesis of Theorem \ref{thm:existenceof2system} holds for $D\rtimes H$. Then the same arguments as in Theorem \ref{thm:QDpcZ2solvableproven} finish the proof. 

Write $B=\prod_j B_j$, where each $B_j$ is an $H$-orbit for the permutation action of $H$ on the direct factors of $B$. For any subgroup $I\leq H$ we have $C_D(I)=C_A(I)\times \prod_j C_{B_j}(I)$. By Proposition \ref{prop:semidirectbasicsCpir}, there are $r$ linearly independent hyperplanes of $H$, $H_1,\ldots,H_r$ such that $|C_D(H_i)|/|C_D(H)|>1$. Hence, for each $i$, either $|C_A(H_i)|/|C_A(H)|>1$ or $|C_{B_j}(H_i)|/|C_{B_j}(H)|>1$ for some $j$. In the former case, let $c'_i\in C_A(H_i)\setminus C_A(H)$ by any element in this non-empty set and set $c_i=(c'_i,1)\in C_D(H_i)\setminus C_D(H)$. To choose $c_i$ for the latter case, consider, for a fixed orbit $B_j$, the following set:
\[
\B_j=\{i\in\{1,\ldots,r\}\text{ such that } |C_{B_j}(H_i)|/|C_{B_j}(H)|>1\}.
\]
Let $B_j={}^{r_1}X\times {}^{r_2}X\times\ldots {}^{r_m}X$, where $X$ is a simple nonabelian $p'$-group and $r_1=1,r_2,\ldots,r_m$ are representatives for the set $H/N_H(X)$ of size $m$. By Proposition \ref{prop:nleqmCFSG}, there are elements $c''_i\in C_{B_j}(H_i)\setminus C_{B_j}(H)$ for each $i\in \B_j$ such that $[c''_i,c''_{i'}]=1$ for all $i,i'\in \B_j$. Then define $c'_i=(1,1,\ldots,1,c''_i,1,\ldots,1)\in C_B(H_i)\setminus C_B(H)$, where we place $c'_i$ in the position of the orbit $B_j$. Define $c_i=(1,c'_i)\in C_D(H_i)\setminus C_D(H)$. It is straightforward that the chosen elements $c_i$'s satisfy the hypothesis of Theorem \ref{thm:existenceof2system}, i.e., they commute pairwise.
\end{proof}

We finally prove Theorem \ref{thm:QCpsolvable}

\begin{thm}\label{thm:QCpsolvableproven}
Quillen's conjecture holds for $p$-solvable groups.
\end{thm}
\begin{proof}
As Hall-Highman \cite[Lemma 1.2.3]{HH1956} is valid for $p$-soluble groups, the argument employed in the proof of Theorem \ref{thm:QCsolvableproven} can be used here but for replacing Theorem \ref{thm:QDpcZ2solvableproven} by Theorem \ref{thm:QDpcZ2p'proven}.
\end{proof}

\begin{finalremark}\label{rmk:final}
Here we give more details on Alperin's approach compared to the approach on this work.  The reductions for the former route are as follows, where $K$ is a $p'$-group:
\[
\xymatrix@R=10pt{
\text{$G$ $p$-solvable }\ar@{~>}[r]^<<<<<<<<<{\eqref{equ:Quillenenoughforpsolvablecase}} & \text{$K\rtimes \FF_p^r$}\ar@{~>}[d]^{\text{CFSG}} \\
\text{$G$ solvable }\ar@{~>}[r]^<<<<<<{\eqref{equ:Quillenenoughforsolvablecase}} & \text{$K\rtimes \FF_p^r$, $K$ solvable } \ar@{~>}[r]^{\text{Alperin}} & \text{Join of $r$ $0$-spheres} 
}
\]
Here, CFSG refers to the argument \eqref{equ:pouterautomorphismsp'nonabeliansimplegroup} already used in the proof of Proposition \ref{prop:nleqmCFSG}. The rightmost reduction on the bottom line consists of Alperin's coprime action arguments leading to a minimal counterexample for which the join of spheres may be considered. Our approach can be summarized as follows:
\[
\xymatrix@R=10pt{
\text{$G$ $p$-solvable }\ar@{~>}[r]^{\eqref{equ:Quillenenoughforpsolvablecase}} & \text{$K\rtimes \FF_p^r$} \ar@{~>}[r]^<<<<<<<<<<{\ref{thm:quotient}} & \text{$K\rtimes \FF_p^r$, $K=A\times S$}\ar@{~>}[rd]^<<<<<<<<{\text{\ref{thm:existenceof2system}$+$CFSG}}\\
&&&\text{\ref{thm:Dsystem}}\\
\text{$G$ solvable }\ar@{~>}[r]^<<<<<<{\eqref{equ:Quillenenoughforsolvablecase}} & \text{$K\rtimes \FF_p^r$, $K$ solvable }\ar@{~>}[r]^<<<<<<{\text{\ref{thm:quotient}}} & \text{$K\rtimes \FF_p^r$, $K=A$}
\ar@{~>}[ru]_<<<<<<<<<<<{\text{\ref{thm:existenceof2system}$+$abelian}}
}
\]
Here, $A$ is an abelian $p'$-group and $S$ is a $p'$-group that is a direct product of nonabelian  simple groups. Theorem \ref{thm:quotient} lifts classes from the quotient, and it is a tool that was not available before. Theorem \ref{thm:existenceof2system} implies the existence of a $2$-system, and it is a generalization to possibly non-solvable groups of the join of spheres construction (see Remark \ref{rmk:liftstoZZjoinofspheres}).
\end{finalremark}

\end{document}